\documentclass[a4paper,12pt,reqno]{amsart}

\pdfoutput=1
\usepackage[T1]{fontenc}
\usepackage{xspace}
\usepackage{amsmath,amssymb,amsthm}
\usepackage{bbm}
\usepackage{color}
\usepackage{enumerate}
\usepackage{stmaryrd} 
\usepackage[toc,page]{appendix} 
\usepackage{a4wide}
\usepackage[pagebackref=False]{hyperref}
\hypersetup{
    bookmarks=true,         
    unicode=false,          
    pdftoolbar=true,        
    pdfmenubar=true,        
    pdftitle={The maximization of the p-Laplacian energy for a two-phase material.},    
    pdfauthor={Casado-Diaz, Conca, Vásquez},     
    pdfkeywords= {two-phase material} {p-Laplacian operator} {relaxation} {smoothness}
{non-existence} 
    colorlinks=true,       
    linkcolor=blue, 
    citecolor=red, 
    filecolor=magenta,      
    urlcolor=cyan        
}

\usepackage{graphicx}
\graphicspath{{figures/}}
\usepackage{tikz}
\usepackage{pgfplots} 
\pgfplotsset{compat=newest}
\usetikzlibrary[pgfplots.groupplots]
\usepgfplotslibrary{groupplots}
\usetikzlibrary{matrix,arrows,decorations.pathmorphing}
\definecolor{k4}{rgb}{0.8,0.8,0.8}
\definecolor{k3}{rgb}{0.6,0.6,0.6}
\definecolor{k2}{rgb}{0.4,0.4,0.4}
\definecolor{k1}{rgb}{0.2,0.2,0.2}

\usepackage{caption}

\usepackage{varwidth}
\DeclareCaptionFormat{myformat}{%
  \begin{varwidth}{\linewidth}%
    \centering
    #1#2#3%
  \end{varwidth}%
}

\numberwithin{equation}{section}

\newtheorem{theorem}{Theorem}[section]
\newtheorem*{thm*}{Theorem}
\newtheorem*{prop*}{Proposition}
\newtheorem{prop}{Proposition}[section]

\newtheorem{lemma}{Lemma}[section]
\newtheorem{defi}{Definition}[section]

\newtheorem{rem}[defi]{Remark}

\newcommand{\weakStar}{\overset{*}{\rightharpoonup}}

\newcommand{\R}{\ensuremath{\mathbb{R}}}

\newcommand{\eps}{\varepsilon}

\newcommand{\norm}[1]{\left\Vert#1\right\Vert}

\newcommand{\be}{\begin{equation}}
\newcommand{\ee}{\end{equation}}
\newcommand{\ba}{\begin{eqnarray}}
\newcommand{\ea}{\end{eqnarray}}
\newcommand{\beq}{\begin{equation}}
\newcommand{\eeq}{\end{equation}}

\renewcommand{\leq}{\leqslant}

\renewcommand{\geq}{\geqslant}

\def \Om {\Omega}
\def \om {\omega}
\def \dis {\displaystyle}

\def\into{\int_\Omega}
\def\Chi{\mathcal{X}}
\def\ep{\varepsilon}
\def\beq{\begin{equation}}
\def\eeq{\end{equation}}
\def\ecart{\noalign{\medskip}}
\def\ba{\begin{array}}
\def\ea{\end{array}}

\author{Juan Casado-Diaz$^{1}$}

\author{Carlos Conca$^{2}$}

\author{Donato Vasquez-Varas$^{3}$}

\address{$^1$\noindent Juan Casado-Díaz
\smallskip
\newline Dpto. de Ecuaciones Diferenciales y Análisis Numérico, Universidad de Sevilla, Sevilla 41012,
Spain (jcasadod@us.es).
\smallskip
\noindent E-mail: \texttt{jcasadod@us.es} 
}
\address{$^2$\noindent Carlos Conca
\smallskip
\newline Department of Engineering Mathematics, Center for Mathematical Modelling (CMM), UMI 2807 CNRS-Chile
\& Center for Biotechnology and Bioengineering (CeBiB), University of Chile.
\smallskip
\noindent E-mail: \texttt{cconca@dim.uchile.cl}
}

\address{$^3$\noindent Donato Vásquez-Varas
\smallskip
\newline Department of Engineering Mathematics, University of Chile.
\newline
\smallskip 
\noindent E-mail: \texttt{dvasquez@dim.uchile.cl}
}

\date{\today}
\title[The maximization of the p-Laplacian energy for a two-phase material]{The maximization of the p-Laplacian energy\\ for a two-phase material} 
\thanks{Juan Casado-Díaz has been partially supported by the Project MTM2017-83583 of the Ministerio de Ciencia, Innovación y Universidades of Spain. Carlos Conca is partially supported by PFBasal-001 and AFBasal170001 projects,
and from the Regional Program STIC-AmSud Project NEMBICA-20-STIC-05.
Donato Vásquez-Varas has been partially supported by the CONICYT PFCHA/DOCTORADO BECAS CHILE/2018 - 21182101}
\begin{document}
\maketitle


\begin{abstract}
{\color{blue} We consider the optimal arrangement of two diffusion materials in a bounded open set $\Om\subset \R^N$ in order to maximize the energy. The diffusion problem is modeled by the $p$-Laplacian operator. It is well known that this type of problems has no solution in general and then that it is necessary to work with a relaxed formulation. In the present paper we obtain such relaxed formulation using the homogenization theory, i.e. we replace both materials by microscopic mixtures of them. Then we get some uniqueness results and a system of optimality conditions. As a consequence we prove some regularity properties for the optimal solutions of the relaxed problem. Namely, we show that the flux is in the Sobolev space $H^1(\Om)^N$ and that the optimal proportion of the materials is derivable in the orthogonal direction to the flux. This will imply that the unrelaxed problem has no solution in general. Our results extend those obtained by the first author for the Laplace operator.}
\end{abstract}


\section{Introduction} 
The present paper is devoted to study  an optimal design problem for a diffusion process in a two-phase material modeled by the $p$-Laplacian operator. Namely, we are interested in the control problem
\begin{equation}
\left\{ \begin{array}{c}
\displaystyle\max_{\om} \int_{\Omega}\big(\alpha\mathcal{X}_{\omega}+\beta\left(1-\mathcal{X}_{\omega}\right)\big)|\nabla u|^{p} dx \\ \ecart\dis -{\rm div}\big((\alpha\mathcal{X}_{\omega}+\beta\left(1-\mathcal{X}_{\omega})\big)|\nabla u|^{p-2}\nabla u\right)=f \ \hbox{ in }\Om\\ \ecart\dis u\in W^{1,p}_{0}(\Omega),\quad
 \omega \subset \Omega \mbox{ mesurable },\quad |\om|\leq \kappa, \end{array}\right.
\label{Problema:Original}
\end{equation}
with $\Omega$ a bounded open set in $\R^N$, $N\geq 2$, $p\in (1,\infty)$,  $\alpha,\beta,\kappa>0$, $\alpha<\beta$, $\mathcal{X}_{\omega}$ the characteristic function of the set $\omega$, and $f\in W^{-1,p'}(\Om)$, with $p'$ is the Holder conjugate of $p$ $\left(p'=\frac{p}{p-1}\right)$.\par

In \eqref{Problema:Original} the equation is understood to hold in the sense of distributions, combined with $u\in W^{1,p}_0(\Om)$, denoting
by $u^\alpha$ and $u^{\beta}$ the values of $u$ in $\om$ and $\Om\setminus \om$ respectively and assuming $\om$ smooth enough, this means that  the interphase conditions on $\partial\om$ are given by
$$u^\alpha=u^\beta,\ \alpha|\nabla u^\alpha|^{p-2}\nabla u^\alpha\cdot\nu=\beta|\nabla u^\beta|^{p-2}\nabla u^\beta\cdot\nu\ \mbox{ on }\partial\om\cap \Om$$
in the sense of the traces in $W^{1/p',p}(\partial\om)$ and $W^{-1/p',p'}(\partial\om)$ respectively. Here $\nu$ denotes a unitary normal vector on $\partial\om$.\par

 Physically the constants $\alpha$ and $\beta$ represent two diffusion materials that we are mixing in order to maximize the corresponding functional, which in \eqref{Problema:Original} represent the potential energy. 
The control variable is the set $\om$ where we place the material $\alpha$. If we do not impose any restriction on the amount of this material, it is simple to check that the solution of \eqref{Problema:Original} is the trivial one given by $\om=\Om$. Thus, the interesting problem corresponds to $\kappa<|\Om|,$ i.e. the material $\alpha$ is better than $\beta$ but it is also more expensive and therefore, we do not want to use a large amount of it in the mixture. The case corresponding to $p=2$  has been studied in several papers (see e.g. \cite{Cas2}, \cite{GoKoRe}, \cite{MuTa}) where some classical applications are the optimal mixture of two materials in the cross-section of a beam in order to minimize the torsion, and 
the optimal arrangement of two viscous fluids in a pipe. For $p\in (1,2)\cup(2,\infty)$ the p-Laplacian operator models the torsional creep in the cross-section of a beam \cite{Kaw} and therefore problem \eqref{Problema:Original} corresponds to find the material which minimizes the torsion for the mixture of  two homogeneous materials in non-linear elasticity.\par
It is well known that a control problem in the coefficients like \eqref{Problema:Original} has no solution in general (\cite{Mur0}, \cite{Mur0b}). In fact, some counterexamples  to the existence of solution for \eqref{Problema:Original} with $p=2$ can be found in \cite{Cas2} and \cite{MuTa}. Thus, it is necessary to work with a relaxed formulation. One way to obtain this  formulation is to use the homogenization theory (\cite{All}, \cite{MuTa}, \cite{Tar}). The idea is to replace the material $\alpha\mathcal{X}_{\omega}+\beta(1-\mathcal{X}_{\omega})$ in \eqref{Problema:Original} by microscopic mixtures of $\alpha,\beta$ with a certain proportion $\theta=\theta(x)\in [0,1]$, $x\in \Omega$. The new materials do not only depend on the proportion of each original material but also on their microscopical distribution. In the case $p=2$, this relaxed formulation has been obtained in \cite{MuTa}. Here we show that a relaxed formulation for \eqref{Problema:Original} is given by
\begin{equation}
\left\{ \begin{array}{c}
 \displaystyle\max_{\theta} \left\{{1\over p}\into \Big(\theta\alpha^{1\over 1-p}+(1-\theta) \beta^{1\over 1-p}\Big)^{1-p}|\nabla u|^pdx\right\} \\ \ecart\dis -{\rm div}\Big(\big(\theta\alpha^{1\over 1-p}+(1-\theta) \beta^{1\over 1-p}\big)^{1-p}|\nabla u|^{p-2}\nabla u\Big)=f \ \hbox{ in }\Om\\ \ecart\dis u\in W^{1,p}_{0}(\Omega),\quad  \theta\in L^\infty(\Om;[0,1]),\quad \into\theta(x)\,dx\leq \kappa, \end{array}\right.
\label{ProRelIn}
\end{equation}
which is equivalent to the Calculus of Variations problem
\begin{equation}
\left\{ \begin{array}{c}
 \displaystyle\min_{\theta} \left\{{1\over p}\into \Big(\theta\alpha^{1\over 1-p}+(1-\theta) \beta^{1\over 1-p}\Big)^{1-p}|\nabla u|^pdx-\langle f,u\rangle\right\} \\ \ecart\dis  u\in W^{1,p}_{0}(\Omega),\quad  \theta\in L^\infty(\Om;[0,1]),\quad \into\theta(x)\,dx\leq \kappa, \end{array}\right.
\label{ProRelIn2}
\end{equation}
\textcolor{red}{where here and in what follows, $\langle f,u\rangle$ denotes the duality product of $f$ and $u$ as elements of $W^{-1,p'}(\Om)$ and $W^{1,p}_0(\Om)$ respectively.}\par
Our main results extend those obtained in \cite{Cas2} (see also \cite{MuTa}) for $p=2$ relative to the uniqueness and regularity of a solution for (\ref{ProRelIn}). Namely, we prove that although it is not clear that (\ref{ProRelIn2}) has a unique solution $(u,\theta)$, the flux
$$\sigma=\Big({\theta\over\alpha^{1\over p-1}}+{1-\theta\over \beta^{1\over p-1}}\Big)^{1-p}|\nabla u|^{p-2}\nabla u$$
is unique. Moreover, assuming $\Om\in C^{1,1}$ and $f\in L^q(\Om)\cap W^{1,1}(\Om)$, with $q>N$, we have that $\sigma$ belongs to $H^1(\Om)^N\cap L^\infty(\Om)$. This is related to some regularity results for the $p$-{L}aplacian operator obtained in \cite{lou2008singular}. We also prove that every solution $(u,\theta)$ of (\ref{ProRelIn2}) satisfies
\beq\label{regusi} u\in W^{1,\infty}(\Om), \quad \partial_i\theta\,\sigma_j-\partial_j\theta\,\sigma_i\in L^2(\Om),\ 1\leq i,j\leq N,\eeq
where $\sigma_i$ denotes the $i$-th component of the vector function $\sigma$, i.e. $\theta$ is derivable in the orthogonal subspace to $\sigma$. The existence of first derivatives for $\sigma$ and $\theta$ will imply that we cannot hope in general an existence result for the unrelaxed problem (\ref{Problema:Original}). Namely, the existence of a solution for (\ref{Problema:Original}) is equivalent to the existence of a solution for (\ref{ProRelIn2}) where $\theta$ only takes the values zero and one, but then the derivatives of $\theta$ in (\ref{regusi}) vanish. Assuming $\Om$ simply connected with connected boundary, we show that  this implies $\sigma=|\nabla w|^{p-2}\nabla w$, with $w$ the unique solution of
$$\left\{\ba{l}\dis -{\rm div}\, (|\nabla w|^{p-2}\nabla w)=f\ \hbox{ in }\Om\\ \ecart\dis w\in W^{1,p}_0(\Om).\ea\right.$$
Similarly to the result obtained in (\cite{Cas2}, \cite{MuTa}), we prove that this is only possible if $\Om$ is a ball.\par
We finish this introduction remembering that the results obtained in the present paper are also related to those given in \cite{Cas} where, for $p=2$, it is considered the minimization in \eqref{Problema:Original} instead of the maximization.  Problem \eqref{Problema:Original} is also related to the minimization of the first eigenvalue for the $p$-Laplacian operator (see \cite{Cas2}, \cite{Cas3}, \cite{CoLaMa}, \cite{CoMaSa}, \cite{MoYou}  for $p=2$), problem which we hope to study in a later work.
\section{Position of the problem. Relaxation and equivalent formulations}
For a bounded open set $\Omega\subset \R^N$, three positive constants $\alpha,\beta,\kappa$ with $0<\alpha<\beta$, $\kappa< |\Om|$, and a distribution  $f\in W^{-1,p'}(\Omega)$, $p>1$, we are interested in the control problem
\begin{equation} 
\left\{\begin{array}{c}
\displaystyle\max_{\om}  \int_{\Omega}\left(\alpha\mathcal{X}_{\omega}+\beta\mathcal{X}_{\Om\setminus\om}\right)|\nabla u_{\omega}|^p dx \\ \noalign{\medskip}
\omega\subset\Omega\mbox{ measurable}, \ |\omega|\leq \kappa \\ \noalign{\medskip}
\displaystyle-{\rm div}\left((\alpha\mathcal{X}_{\omega}+\beta\mathcal{X}_{\Om\setminus\om})|\nabla u_{\omega}|^{p-2}\nabla u_\om\right)=f \mbox{ in }\Omega, \ u_{\omega}\in W^{1,p}_0(\Omega).
\end{array}\right.
\label{Problem:Compliance}
\end{equation}
Here $\alpha$ and $\beta$ represent the  diffusion coefficients of two materials, where the diffusion process is modeled by the $p$-Laplacian operator. The problem consists in maximizing the potential energy. 
\par
 Using $u_\om$ as test function in the state equation we have
$$\displaystyle\int_{\Omega}\left(\alpha\mathcal{X}_{\omega}+\beta\mathcal{X}_{\Om\setminus\om}\right)|\nabla u_{\omega}|^p dx=\langle f,u_{\omega}\rangle,$$
 By the above equality and since $p'=\frac{p}{p-1}$ we have 
$$\begin{array}{l}
\displaystyle\int_{\Omega}\left(\alpha\mathcal{X}_{\omega}+\beta\mathcal{X}_{\Om\setminus\om}\right)|\nabla u_{\omega}|^p dx\\=\displaystyle -p'\left(\frac{1}{p}\int_{\Omega}\left(\alpha\mathcal{X}_{\omega}+\beta\mathcal{X}_{\Om\setminus\om}\right)|\nabla u_{\omega}|^p dx-\int_{\Omega}\left(\alpha\mathcal{X}_{\omega}+\beta\mathcal{X}_{\Om\setminus\om}\right)|\nabla u_{\omega}|^p dx\right)\\ 
\displaystyle  =-p'\left(\frac{1}{p}\int_{\Omega}\left(\alpha\mathcal{X}_{\omega}+\beta\mathcal{X}_{\Om\setminus\om}\right)|\nabla u_{\omega}|^p dx-\left<f,u_{\omega}\right>\right
)\end{array}$$ 
which combined with $u_\om$, unique solution of the minimization problem
$$\min_{u\in W^{1,p}_0(\Om)}\left\{{1\over p}\into \big(\alpha\mathcal{X}_\om+\beta\Chi_{\Om\setminus\om}\big)|\nabla u|^pdx-\langle f,u\rangle\right\},$$
gives the equivalent formulation for problem \eqref{Problem:Compliance}:
\begin{equation} 
\left\{\begin{array}{c}
\displaystyle\min_{\om,u}  \left\{{1\over p}\into \big(\alpha\Chi_\om+\beta\Chi_{\Om\setminus\om}\big)|\nabla u|^pdx-\langle f,u\rangle\right\} \\ \noalign{\medskip}
u\in W^{1,p}_0(\Om),\quad \om\subset\Omega\mbox{ measurable}, \quad  |\omega|\leq \kappa.
\end{array}\right.
\label{CompMin}
\end{equation}
\par
It is known that the maximum in
 \eqref{Problem:Compliance} or the minimun in \eqref{CompMin} are not achieved, i.e., that \eqref{Problem:Compliance} (or \eqref{CompMin})  has  no solution in general. Namely, for $p=2$ and $f=1$, it has been proved in \cite{Cas2} and \cite{MuTa}  that if $\Om$ is smooth, with  connected smooth boundary, and  \eqref{Problem:Compliance} has a solution, then $\Om$ is a ball. Some other classical counterexamples to the existence of solution for problems related to \eqref{Problem:Compliance} can be found in \cite{Mur0} and \cite{Mur0b}. Due to this difficulty it is then necessary to find a relaxed formulation for \eqref{Problem:Compliance}. This is done by the following theorem
\begin{theorem} \label{ThRelax} A relaxed formulation of problem \eqref{CompMin} is given by
\begin{equation} 
\left\{\begin{array}{c}\dis
\min_{\theta,u} \left\{ {1\over p}\into \Big(\theta\alpha^{1\over 1-p}+(1-\theta)\beta^{1\over 1-p}\Big)^{1-p} |\nabla u|^pdx-\langle f,u\rangle\right\} \\ \noalign{\medskip}
\displaystyle u\in W^{1,p}_0(\Om),\quad \theta\in L^{\infty}(\Omega;[0,1]), \quad \int_{\Omega}\theta dx\leq \kappa,\end{array}\right.
\label{Problem:ComplianceRelaxtion}
\end{equation}
in the following sense:
\begin{enumerate}
\item Problem \eqref{Problem:ComplianceRelaxtion} has a solution.
\item The infimum for problem \eqref{CompMin} agrees with the minimum for \eqref{Problem:ComplianceRelaxtion}.
\item Every minimizing sequence $(u_n,\om_n)$ for \eqref{CompMin} has a subsequence still denoted by $(u_n,\om_n)$ such that
\beq\label{convsumi} u_n\rightharpoonup u\ \hbox{ in }W^{1,p}_0(\Om),\quad\Chi_{\om_n}\stackrel{\ast}\rightharpoonup \theta\ \hbox{ in }L^\infty(\Om),\eeq
with $(u,\theta)$ solution of \eqref{Problem:ComplianceRelaxtion}.
\item For every pair $(u,\theta)\in W^{1,p}_0(\Om)\times L^\infty(\Om;[0,1])$  there exist $u_n\in W^{1,p}_0(\Om)$, $\om_n\subset\Om$ measurable, with $|\om_n|\leq\kappa$ such that \eqref{convsumi} holds and such that
\beq\label{alcmin}\lim_{n\to\infty}\into \big(\alpha\Chi_{\om_n}+\beta\Chi_{\Om\setminus\om_n}\big)|\nabla u_n|^pdx=\into \Big(\theta \alpha^{1\over 1-p}+(1-\theta)\beta^{1\over 1-p}\Big)^{1-p}|\nabla u|^pdx.\eeq
\end{enumerate}
\end{theorem}
\begin{rem} Such as we will see in the proof of Theorem \ref{ThRelax}, the relaxed materials  in \eqref{Problem:ComplianceRelaxtion} are obtained as a simple lamination in a parallel direction to $\nabla u$. In this context, a laminated material corresponds to a particular distribution of two materials, which depends exclusively on one direction, say $\xi\in\R^{N}$, which is represented by a function $\varphi\in L^{\infty} (\Omega; [0, 1])$ with a generic form as follows:
\[
\varphi (x) = g (\xi \cdot x) \quad   \forall x\in \Om, 
\]
where $g$ is a real-valued function. (see sections 2.3.5 and 2.2.1 in \cite{All} for more details on laminated materials). 
\end{rem}
\begin{proof} [Proof of Theorem \ref{ThRelax}.] Using that the function $J:\R^N\times (0,\infty)\to \R$ defined by
\beq\label{convfun}J(\xi,t)={|\xi|^p\over t^{p-1}},\quad\forall\,(\xi,t)\in\R^N\times (0,\infty),\eeq
is convex, and the sequential compactness of the bounded sets in $W^{1,p}_0(\Om)\times L^\infty(\Om)$ with respect to the weak-$\ast$ topology, it is immediate to show that \eqref{Problem:ComplianceRelaxtion} has at least a solution and that every minimizing sequence $(u_n,\theta_n)$ for \eqref{Problem:ComplianceRelaxtion} has a subsequence which converges in $W^{1,p}_0(\Om)\times L^\infty(\Om)$ weak-$\ast$ to a minimum.\par
Since problem \eqref{CompMin} consists in minimizing the same functional than the one in \eqref{Problem:ComplianceRelaxtion}, but on the smaller set
$$\Big\{(u,\Chi_\om)\in W^{1,p}_0(\Om)\times L^\infty(\Om;[0,1]):\ \om\subset\Om,\ \into\Chi_\om\,dx\leq\kappa\Big\},$$
it is clear that the infimum in \eqref{CompMin} is bigger or equal than the minimum in \eqref{Problem:ComplianceRelaxtion}. Thus, 
taking into account that the convergence of the minimizing sequences stated above will imply statement (3), we 
deduce that it is enough to prove statement (4) to complete the proof of Theorem \ref{ThRelax}. For this purpose, we introduce the functions (the index $\sharp$ means periodicity) $H\in L^\infty((0,1)\times \R)\cap C^0([0,1];L^1_\sharp(0,1))$, $G\in W^{1,\infty}((0,1)\times \R)\cap
C^0([0,1];W^{1,1}_\sharp (0,1))$,  by
\beq\label{DefHqr} H(q,r)=\sum_{k=-\infty}^\infty \Chi_{[k,k+q)}(r),\ \ G(q,r)=qr-\int_0^rH(q,s)\,ds,\quad\forall\, q,r\in [0,1]\times \R.
\eeq\par
Now, for a pair
$(u,\theta)\in C^1_c(\Om)\times C^0(\overline\Om)$ with
$$\into \theta\,dx< \kappa,$$
and $\delta>0$, we consider a family of cubes $Q_i$, $1\leq i\leq n_\delta$, of side $\delta$ such that
$$\overline\Om\subset \bigcup_{i=1}^{n_\delta}Q_i,\quad |Q_i\cap Q_j|=0,\ \hbox{ if }i\not =j,$$
and a partition of the unity in $\overline\Om$ by functions $\psi_i\in C^\infty_c(\R^N)$, with 
$${\rm sup}(\psi_i)\subset Q_i+B(0,\delta),\ \psi_{i}(x) \geq 0,\ 1\leq i\leq n_\delta \mbox{ and }  \sum_{i=1}^{n_{\delta}}\psi_{i}(x)=1,\ \forall x \in \Omega .$$ 
{Then, we take
$$q_i={1\over \delta^N}\int_{Q_i}\theta\,dx,\qquad \xi_i={1\over \delta^N}\int_{Q_i}\nabla u\,dx,\quad \zeta_i=\left\{\ba{ll}\xi_i & \hbox{ if }\xi_i\not=0\\ \ecart\dis
e & \hbox{ if }\xi_i=0,\ea\right.$$
with $e\in\R^N\setminus\{0\}$ fixed, and we introduce, for every $\ep>0$, the sets $\om_{\delta,\ep}\subset\Om$ and the functions $u_{\delta,\ep}\in W^{1,\infty}(\Om)$, with compact support by
$$\Chi_{\om_{\delta,\ep}}=\sum_{i=1}^{n_\delta} H\Big(q_i,{\zeta_i\cdot x\over \ep}\Big)\Chi_{Q_i},\quad u_{\delta,\ep}=u+\ep\sum_{i=1}^{n_\delta}\psi_i
{G\big(q_i,{\xi_i\cdot x\over \ep}\big)\big(\beta^{1\over 1-p}-\alpha^{1\over 1-p}\big)
\over \alpha^{1\over 1-p}q_i+\beta^{1\over 1-p}(1-q_i)}.$$}
{Using the result (see e.g. \cite{Allts})
\beq\label{thccp}\Phi\big(x,{x\cdot\xi\over \ep}\big)\stackrel{\ast}\rightharpoonup \int_0^1\Phi(x,s)\,ds\ \hbox{ in }L^\infty(\Om),
\eeq
for every $\Phi\in C^0(\overline\Om;L^1_\sharp(0,1))\cap L^\infty(\Om\times \R)$ and every $\xi\in\R^N\setminus\{0\}$,
we have that $\om_{\delta,\ep}$ satisfies
\beq\label{convoed}\Chi_{\om_{\delta,\ep}}\stackrel{\ast}\rightharpoonup \theta_\delta:=\sum_{i=1}^{n_\delta}q_i\Chi_{Q_i}\ \hbox{ in }L^\infty(\Om),\ \hbox{ when }\ep\to 0,\eeq
where thanks to $\theta$ uniformly  continuous, we also have
\beq\label{contd}\theta_\delta\to \theta\ \hbox{ in }L^\infty(\Omega;[0,1]),\ \hbox{ when }\delta\to 0.\eeq
In particular, since the integral of $\theta$ is strictly smaller than $\kappa$, we deduce that  for every $\delta>0$ small enough, there exists $\ep_\delta>0$ such that 
\beq\label{acotmeoed} |\om_{\delta,\ep}|<\kappa,\quad\forall\, 0<\ep<\ep_{\delta}.\eeq
Since $q(q-1)\leq G(q,r)\leq 0$, for every $q\in [0,1]$ and every $r\in \R$,  we also have
the existence of $C>0$ such that
\beq\label{convude}\|u_{\delta,\ep}-u\|_{C^0(\overline\Om)}\leq  C\ep,\qquad \forall\,\ep,\delta>0\eeq
and taking into account that $u$ has compact support and  that $G(q,0)=0$, we deduce that, for $\delta$ small enough,  $u_{\delta,\ep}$  has compact support and thus belongs to
$W^{1,p}_0(\Om)$. Moreover, thanks to (\ref{thccp}) (observe that there is not problem if $\xi_i=0$ because then $G(q_i,{\xi_i\cdot x\over\ep})=0$ for every $x\in\R^N$)
$$\ba{ll}\dis \nabla u_{\delta,\ep} &\dis =\nabla u+\sum_{i=1}^{n_\delta}
{\big(\beta^{1\over 1-p}-\alpha^{1\over 1-p}\big)
\over \alpha^{1\over 1-p}q_i+\beta^{1\over 1-p}(1-q_i)}\Big(\ep\nabla\psi_i G\big(q_i,{\xi_i\cdot x\over \ep}\big)+
\psi_i\big(q_i-H\big(q_i,{\xi_i\cdot x\over \ep}\big)\big)\xi_i\Big)\\ 
\ecart &\dis
\stackrel{\ast}\rightharpoonup \nabla u\ \hbox{ in }L^\infty(\Om)\ \hbox{ when }\ep\to 0,\quad \forall\,\delta>0.\ea$$
Therefore
\beq\label{conued} u_{\delta,\ep}\stackrel{\ast}\rightharpoonup u\ \hbox{  in }W^{1,\infty}(\Om)\cap W^{1,p}_0(\Om)\ \hbox{ when }\ep\to 0,\quad\forall\,\delta>0\ \hbox{ small engouh}.\eeq} \par
{On the other hand, using the above expression of $\nabla u_{\delta,\ep}$, and denoting $H_i(s)=H(q_i,s)$, we can use (\ref{thccp}) combined with $H(q,s)=1$ if $s\in (0,q)$, $H(q,s)=0$ if $s\in (q,1)$, and $\xi_i=0$ is $\zeta_i\not =\xi_i$ to deduce
$$\ba{l}\displaystyle \lim_{\ep\to 0}\into \big(\alpha\Chi_{\om_{\delta,\ep}}+\beta(1-\Chi_{\om_{\delta,\ep}})\big)|\nabla u_{\delta,\ep}|^pdx\\ 
\ecart
\displaystyle =\sum_{i=1}^{n_\delta}\int_{Q_i}\int_0^1\big(\alpha H_i(s)+\beta(1-H_i(s))\big)\Bigg|\nabla u+\sum_{i=1}^{n_\delta}\psi_i
{\big(q_i-H_i(s)\big)\big(\beta^{1\over 1-p}-\alpha^{1\over 1-p}\big)
\over \alpha^{1\over 1-p}q_i+\beta^{1\over 1-p}(1-q_i)}\xi_i\Bigg|^pdsdx\\ \ecart\dis
=\sum_{i=1}^{n_\delta}\int_{Q_{i}} \alpha q_i\Bigg|\nabla u+
{(q_i-1)\big(\beta^{1\over 1-p}-\alpha^{1\over 1-p}\big)
\over \alpha^{1\over 1-p}q_i+\beta^{1\over 1-p}(1-q_i)}\xi_i\Bigg|^pdx\\ \ecart\dis
+\sum_{i=1}^{n_\delta}\int_{Q_{i}} \beta(1- q_i)\Bigg|\nabla u+
{q_i\big(\beta^{1\over 1-p}-\alpha^{1\over 1-p}\big)
\over \alpha^{1\over 1-p}q_i+\beta^{1\over 1-p}(1-q_i)}\xi_i\Bigg|^pdx.\ea$$}
{Thanks to the uniform continuity of $\theta $ and $\nabla u$, we can also take the limit when $\delta$ tends to zero in the right-hand side of the above equality to get
\beq\label{paledt0}\ba{l}\displaystyle \lim_{\delta\to 0}\Bigg(\sum_{i=1}^{n_\delta}\int_{Q_{i}} \alpha q_i\Bigg|\nabla u+
{(q_i-1)\big(\beta^{1\over 1-p}-\alpha^{1\over 1-p}\big)
\over \alpha^{1\over 1-p}q_i+\beta^{1\over 1-p}(1-q_i)}\xi_i\Bigg|^pdx\\ \ecart\dis\qquad
+\sum_{i=1}^{n_\delta}\int_{Q_{i}} \beta(1- q_i)\Bigg|\nabla u+
{q_i\big(\beta^{1\over 1-p}-\alpha^{1\over 1-p}\big)
\over \alpha^{1\over 1-p}q_i+\beta^{1\over 1-p}(1-q_i)}\xi_i\Bigg|^pdx\Bigg)\\ \ecart\dis = 
\int_{\Om}\Bigg(\alpha\theta\Big|1+\frac{(\theta-1)(\beta^{\frac{1}{1-p}}-\alpha^{\frac{1}{1-p}})}{\alpha^{\frac{1}{1-p}}\theta+\beta^{\frac{1}{1-p}}(1-\theta)}\Big|^{p}\\ \ecart\dis \qquad
+\beta(1-\theta)\Big|1+\frac{\theta(\beta^{\frac{1}{1-p}}-\alpha^{\frac{1}{1-p}})}{\alpha^{\frac{1}{1-p}}\theta+\beta^{\frac{1}{1-p}}(1-\theta)}\Big|^{p}\Bigg)|\nabla u|^{p}dx\\ \ecart\dis
=\into \Big(\theta\alpha^{1\over 1-p}+(1-\theta) \beta^{1\over 1-p}\Big)^{1-p} |\nabla u|^pdx.\ea\eeq}\par
{Let us now use that  for $\ep<1$,  $\nabla u_{\delta,\ep}$ is bounded in $L^\infty(\Om)^N$, independently of $\delta$ and $\ep$, and $\chi_{\om_{\delta,\ep}}\in\{0,1\}$. Thus, there exists $C\geq 1$  such that
$$\|\Chi_{\om_{\delta,\ep}}\|_{L^\infty(\Om)}\leq 1,\quad \|\partial_ju_{\delta,\ep}\|_{L^\infty(\Om)}\leq C,\ 1\leq j\leq N,\qquad\forall\,\ep,\delta>0,\ 0<\ep<1.$$
Here, we recall that the closed ball $\overline B_C$ of center 0 and radius $C$ in $L^\infty(\Om)$, endowed with 
the weak-$\ast$ topology is metrizable. Taking $d$ a suitable distance, and using (\ref{convoed}), (\ref{acotmeoed}) and (\ref{conued}), we can choose for every $\delta>0$, $\ep(\delta)>0$ such that
$$d(\Chi_{\om_{\delta,\ep(\delta)}},\theta_\delta)<\delta,\quad |\om_{\delta,\ep(\delta)}|<\kappa,\quad d(\partial_ju_{\delta,\ep(\delta)},\partial_ju)<\delta,\ 1\leq j\leq N,$$
\beq \left|\begin{array}{l}
\displaystyle\int_{\Omega} \big(\alpha\Chi_{\om_{\delta,\ep(\delta)}}+\beta(1-\Chi_{\om_{\delta,\ep(\delta)}})\big)|\nabla u_{\delta,\ep(\delta)}|^pdx \\
\displaystyle -\sum_{i=1}^{n_\delta}\into\alpha q_i\Bigg|\nabla u+
{(q_i-1)\big(\beta^{1\over 1-p}-\alpha^{1\over 1-p}\big)
\over \alpha^{1\over 1-p}q_i+\beta^{1\over 1-p}(1-q_i)}\xi_i\Bigg|^pdx\\
\displaystyle -\sum_{i=1}^{n_\delta}\into\beta(1- q_i)\Bigg|\nabla u+
{q_i\big(\beta^{1\over 1-p}-\alpha^{1\over 1-p}\big)
\over \alpha^{1\over 1-p}q_i+\beta^{1\over 1-p}(1-q_i)}\xi_i\Bigg|^pdx 
\end{array}\right|< \delta.\eeq}
{Then, taking into account (\ref{contd}) and (\ref{paledt0}), we get
$$ \Chi_{\om_{\delta,\ep(\delta)}}\weakStar \theta\mbox{ in }L^{\infty}(\Omega),\quad |\om_{\delta,\ep(\delta)}|<\kappa,\quad u_{\delta,\eps(\delta)}\weakStar u\mbox{ in }W^{1,\infty}(\Omega)\cap W^{1,p}_0(\Om),$$
$$ \lim_{\delta\to 0}\into \big(\alpha\Chi_{\om_{\delta,\ep(\delta)}}+\beta(1-\Chi_{\om_{\delta,\ep(\delta)}}\big)|\nabla u_{\delta,\ep(\delta)}|^pdx = \into \Big(\theta\alpha^{1\over 1-p}+(1-\theta) \beta^{1\over 1-p}\Big)^{1-p} |\nabla u|^pdx.$$}\par
{This proves assertion (4) for  $u$, $\theta$ smooth and $\into \theta\,dx<\kappa$. The general result follows by density.}\end{proof}
\begin{rem}
We can express problem (\ref{Problem:ComplianceRelaxtion}) in a simpler way defining 
\beq\label{defc} c:=\Big(\frac{\beta}{\alpha}\Big)^{1\over p-1}\hskip-5pt -1>0, \qquad\tilde{f}:=\frac{f}{\beta},\eeq
which provides
\begin{equation}
\left\{ \begin{array}{c}
 \displaystyle\min_{\theta,u}\left\{\frac{1}{p}\int_{\Omega}\frac{|\nabla u|^{p}}{(1+c\,\theta)^{p-1}}dx-<\tilde f,u>\right\} \\  \ecart
\displaystyle  u\in W^{1,p}_{0}(\Omega),\quad \theta\in L^{\infty}(\Omega;[0,1]), \quad \int_{\Omega}\theta dx\leq \kappa.
\end{array}  \right.
\label{Relaxed problem:min}
\end{equation}
For simplicity, in the following we will redefine $f$ as $\tilde{f}$.\end{rem}
\section{Uniqueness results and optimality conditions for the relaxed problem}
Since in problem \eqref{Relaxed problem:min} the cost functional is not  strictly convex, the uniqueness of solution is not clear. However, let us prove in Proposition \ref{Teo:def:flux} that the flux 
\begin{equation}
\hat{\sigma}:=\frac{|\nabla \hat{u}|^{p-2}}{(1+c\,\hat\theta)^{p-1}}\nabla \hat{u},
\label{def:flux}
\end{equation}
with $(\hat u,\hat \theta)$ a solution of   \eqref{Relaxed problem:min}  is uniquely defined. The result follows from a dual formulation of \eqref{Relaxed problem:min} as a min-max problem. In the case $p=2$, a similar result has been obtained in \cite{MuTa}.
\begin{prop}
\label{Teo:def:flux}
For every solution $(\hat{u},\hat \theta)\in W^{1,p}_{0}(\Omega)\times L^{\infty}(\Omega;[0,1])$ of \eqref{Relaxed problem:min},  the flux $\hat{\sigma}$ defined by \eqref{def:flux}
is the unique solution of 
\begin{equation}
\min_{\begin{array}{c}
{\scriptstyle -{\rm div}\,\sigma=f }
\\ 
{\scriptstyle\sigma \in L^{p'}(\Omega)^N}
\end{array}}\max_{\begin{array}{c}
{\scriptstyle\theta \in L^{\infty}(\Omega;[0,1])} \\  
{\scriptstyle\int_{\Omega}\theta\,dx\leq \kappa}
\end{array}} \int_{\Omega}(1+c\,\theta)|\sigma|^{p'}dx.
\label{problem:flux min-max problem}
\end{equation}
The function  $\hat{\theta}$ solves the problem 
\begin{equation}
\max_{\begin{array}{c}
{\scriptstyle\theta \in L^{\infty}(\Omega;[0,1]) } \\ 
{\scriptstyle\int_{\Omega}\theta\,dx\leq \kappa}
\end{array}}\min_{\begin{array}{c}
{\scriptstyle-{\rm div}\sigma=f}
\\ 
{\scriptstyle\sigma \in L^{p'}(\Omega)^N} 
\end{array}} \int_{\Omega}(1+c\,\theta)|\sigma|^{p'}dx,
\label{problem:flux max-min problem}
\end{equation} 
and the minimum value in  \eqref{problem:flux min-max problem} agrees with the maximum in \eqref{problem:flux max-min problem}.
\end{prop}
\begin{proof} For $\theta\in L^\infty(\Om;[0,1])$, we define $\sigma_\theta\in L^{p'}(\Omega)^{N}$ as the unique solution of
\[\min_{\begin{array}{c}
{\scriptstyle -{\rm div}\sigma=f}
\\ 
{\scriptstyle\sigma \in L^{p'}(\Omega)^N}
\end{array}} \int_{\Omega}(1+c\,\theta)|\sigma|^{p'}dx.\]
The uniqueness of $\sigma_{\theta}$ is ensured by the strictly convexity of the problem. Then, taking into account that $\sigma_\theta$ satisfies
$$p'\into (1+c\theta)|\sigma_\theta|^{p'-2}\sigma_\theta\cdot \eta\,dx=0,\quad \forall\, \eta\in L^{p'}(\Om),\ \hbox{ with }{\rm div}\,\eta=0,$$
we deduce the existence of $u_\theta\in W^{1,p}_0(\Om)$ such that $(1+c\theta)|\sigma_\theta|^{p'-2}\sigma_\theta=\nabla u_\theta$ in $\Om$. Using also that $-{\rm div}\,\sigma_\theta=f$ in $\Om$, we get that $u_\theta$ is the unique solution of
$$-{\rm div}\left({|\nabla u_\theta|^{p-2}\over (1+c\theta)^{p-1}} \nabla u_\theta\right)=f\ \hbox{ in }\Om,\quad u_\theta\in W^{1,p}_0(\Om),$$
or equivalently, of the minimization problem
$$\min_{u\in W_{0}^{1,p}(\Omega)}\left\{\frac{1}{p} \int_{\Omega}\frac{|\nabla u|^{p}}{(1+c\,\theta)^{p-1}}dx-\langle f,u\rangle\right\},$$
which combined with
$$\frac{1}{p} \int_{\Omega}\frac{|\nabla u_\theta|^{p}}{(1+c\,\theta)^{p-1}}dx-\langle f,u_\theta \rangle=-{1\over p'} \int_{\Omega}(1+c\,\theta)|\sigma_\theta|^{p'}dx,$$
proves that $(\hat u,\hat\theta)$ is a solution of \eqref{Relaxed problem:min} if and only if $\hat \theta$ is a solution of the {max-min} problem \eqref{problem:flux max-min problem}, and $(\hat\theta,\hat\sigma)$, with $\hat \sigma$ defined by \eqref{def:flux},  is a saddle point. From the von Neumann Min-Max Theorem  { \cite[Theorem~2.G and Proposition 1 in Chapter 2]{Zeidler}}, we get that the minimum in \eqref{problem:flux min-max problem} agrees with the maximum in  \eqref{problem:flux max-min problem}, and
 that $\hat\sigma$ is a solution of (\ref{problem:flux min-max problem}).
 Taking into account that the functional 
\begin{align*}
\sigma\in L^{p'}(\Om)^N\mapsto \max_{\begin{array}{c}
{\scriptstyle \theta \in L^{\infty}(\Omega;[0,1])} \\  
{\scriptstyle \int_{\Omega}\theta dx\leq \kappa}
\end{array}} \int_{\Omega}(1+c\,\theta)|\sigma|^{p'}dx
\end{align*}
is strictly convex, as a maximum of a family of strictly convex functions, we deduce the uniqueness of  $\hat{\sigma}$. 
\end{proof}\par\noindent
The following  theorem provides a system of optimality conditions for the convex problem \eqref{Problem:ComplianceRelaxtion}. It proves in particular that $\hat u$ is the solution of a nonlinear Calculus of Variations problem  which does not contain the proportion $\hat \theta$. We refer to Section 4 in \cite{GoKoRe} for a related result in the case $p=2$.
\begin{theorem}
\label{Teo:problema:solou} A pair
 $(\hat u,\hat\theta)\in W^{1,p}_{0}(\Omega)\times L^{\infty}(\Omega;[0,1])$ is a solution of \eqref{Relaxed problem:min} if and only if there exists  $\hat\mu\geq 0$ such that $\hat u$ is a solution of
 \begin{equation}
 \min_{u\in W_{0}^{1,p}(\Omega) }\left(\int_{\Omega}F(|\nabla u|)dx-\left<f,u\right> \right),
 \label{Problema:solou}
 \end{equation}
with $F\in C^1([0,\infty))\cap W^{2,\infty}_{loc}(0,\infty)$, the convex function defined by 
 \begin{equation} \label{Definicion:F} 
F(0)=0,\qquad  F'(s)=\left\{\begin{array}{cl}
s^{p-1} & \mbox{ if }\ 0\leq s < \hat{\mu}  \\ \ecart
\hat{\mu}^{p-1} & \mbox{ if }\ \hat{\mu} \leq s \leq (1+c)\hat\mu\\ \ecart
\displaystyle {s^{p-1}\over (1+c)^{p-1}} & \mbox{ if }\ (1+c)\hat\mu<s,
\end{array}  \right.
 \end{equation}
and $\hat\mu$, $\hat\theta$ are related by
\begin{itemize}
\item If $\hat\mu=0$ then
 \beq\label{condmu0} \hat\theta=1\ \hbox{ a.e. in }\ \big\{|\nabla\hat u|>0\big\},\quad \into\hat \theta\,dx\leq \kappa.\eeq
 \item If $\hat{\mu}>0$, then 
 \begin{equation}
\hat{\theta}=\left\{\begin{array}{cl}
0 & \mbox{ if }\ 0\leq |\nabla \hat{u}|<\hat\mu\\ \ecart\displaystyle
\frac{1}{c}\left({|\nabla\hat{u}|\over\hat\mu}-1\right) & \mbox{ if }\ \hat \mu \leq |\nabla \hat{u}| < (1+c)\hat \mu \\ \ecart
1 & \mbox{ if }\ (1+c)\hat \mu<|\nabla \hat{u}|,
 \end{array}\right. \qquad \qquad\into\hat\theta\,dx=\kappa.
 \label{characterization:thetha:mupos}\end{equation}
 \end{itemize}
\end{theorem}
\begin{proof}  Applying Kuhn-Tucker's theorem to the convex problem \eqref{Problem:ComplianceRelaxtion}, we get that
$(\hat u,\hat\theta)$ is a solution if and only if there exists  $\hat\mu\geq 0$ such that $(\hat u,\hat \theta)$ solves
\begin{equation} \label{KTRelaxedp}
\min_{\ba{c}{\scriptstyle u\in W^{1,p}_0(\Om)}\\{\scriptstyle\theta\in L^\infty(\Om;[0,1])}\ea} \left\{\into\Big({1\over p}\frac{|\nabla u|^{p}}{(1+c\,\theta)^{p-1}}+{c\hat\mu^p\over p'}\theta\Big)dx-<f,u>\right\},\end{equation}
and
\beq \label{KTRelaxedp2}
\into\hat \theta\,dx\leq \kappa,\qquad \hat \mu\left(\into\hat\theta\,dx-\kappa\right)=0.\eeq
Differentiating in (\ref{KTRelaxedp}) we have that $(\hat u,\hat\theta)$ is a solution of (\ref{KTRelaxedp}) if and only if 
\beq\label{ecuauh} \into {|\nabla \hat u|^{p-2}\nabla \hat u\cdot\nabla \hat v\over (1+c\hat\theta)^{p-1}}dx=\langle f,v\rangle,\quad\forall\, v\in W^{1,p}_0(\Om),\eeq
\beq\label{ecuath}
\into \Big(\hat\mu^p-{|\nabla\hat u|^p\over (1+c\hat\theta)^p}\Big)\big(\theta-\hat\theta\big)dx\geq 0,\quad\forall\, \theta\in L^\infty(\Om;[0,1]).\eeq
Condition \eqref{ecuauh} is equivalent to $\hat u$ solution of the minimum problem
 \begin{equation}\label{Promiuco}
 \min_{u\in W_{0}^{1,p}(\Omega) }\left\{{1\over p}\int_{\Omega}{|\nabla u|^p\over (1+c\hat\theta)^{p-1}}dx-\left<f,u\right> \right\},
 \eeq
while \eqref{ecuath} is equivalent to $\hat\theta$ satisfying \eqref{condmu0} or \eqref{characterization:thetha:mupos} depending on whether $\hat\mu=0$ or  $\hat\mu>0$.
Replacing this value of $\hat\theta$ in \eqref{KTRelaxedp} we have the equivalence between \eqref{Promiuco} and \eqref{Problema:solou}. 
\end{proof}
\begin{rem}\label{teo:characterization:theta} Using  \eqref{condmu0} or \eqref{characterization:thetha:mupos} and expression \eqref{def:flux} of $\hat\sigma$, we have that $\hat\theta$ satisfies
\begin{equation}
\hat{\theta}(x)=\left\{\begin{array}{cc}
1 & \mbox{ if }\ |\hat{\sigma}|>\hat{\mu}\\ \ecart
0 & \mbox{ if }\ |\hat{\sigma}|<\hat{\mu}.
\end{array}\right.
\label{characterization:theta}
\end{equation} 
Moreover, Theorem \ref{Teo:problema:solou} implies $\hat\mu=0$ if and only if the unique solution $\tilde u$ of
$$ \min_{u\in W_{0}^{1,p}(\Omega) }\left\{{1\over p}\into {|\nabla u|^p\over (1+c)^{p-1}}dx-\left<f,u\right> \right\},
$$
satisfies
$$\big|\{x\in\Om:\ |\nabla \tilde u|>0\big\}\big|\leq \kappa,$$
where in this case $\hat u=\tilde u$.
\end{rem}
\section{Regularity for the relaxed problem}
In the present section we study the regularity of the solutions of problem \eqref{Relaxed problem:min}. As a consequence we show that the unrelaxed problem \eqref{CompMin} has no solution in general. We begin by stating the main results. The corresponding proofs are given later.
\begin{theorem} \label{ThprReg}
Let $\Omega\subset\R^N$ be a $C^{1,1}$ bounded  open set  and $(\hat u,\hat\theta)\in W^{1,p}_{0}(\Omega)\times L^{\infty}(\Omega;[0,1])$ be a solution of \eqref{Relaxed problem:min}, then, for  $\hat\sigma$ defined by \eqref{def:flux} and  $\hat\mu$ given by Theorem \ref{Teo:problema:solou} we have: 
\begin{enumerate}
\item If $f\in W^{-1,q}(\Omega)$, $p'\leq q<\infty$, then $\nabla\hat{u}\in L^{q(p-1)}(\Omega)^N$ and  there exists $C>0,$ which only depends on $p,q,N$ and $\Omega$ such that 
\begin{equation}
\label{teo:reg:cond1}
\norm{\nabla \hat{u}}_{L^{q(p-1)}(\Omega)^N}\leq C\big(\|f\|_{W^{-1,q}(\Om)}^{1\over p-1}+\hat \mu\big).
\end{equation}
\item If $f\in L^{q}(\Omega)$ with $q>N,$  then there exists $C>0$ which only depends  on $p,q,N$ and $\Om$ such that 
\begin{equation}
\norm{\nabla \hat{u}}_{L^{\infty}(\Omega)^N}\leq C\big(\|f\|_{L^q(\Om)}^{1\over p-1}+\hat \mu\big).\label{teo:reg:cond2}
\end{equation} 
\item If $f\in W^{1,1}(\Om)\cap L^{2(1+r)}(\Om),$ with $r\geq 0$ or $f\in W^{1,2(1+r)}(\Om)$ with $r\in (-1/2,0)$, then the function $|\hat\sigma|^r\hat\sigma$ is in $H^1(\Om)^N$ and there exists $C>0$, which only depends on $p,q,N,\hat\mu$ and $\Om$ such that
\begin{equation}
\label{teo:reg:cond3}
\big \||\hat\sigma|^r\sigma\big\|_{H^1(\Omega)^N}\leq\left\{\ba{ll}\dis  C\left(1+\norm{f}_{W^{1,1}(\Omega)}+\norm{f}_{L^{2(1+r)}(\Omega)}^{2(1+r)}\right)&\hbox{if } r\geq 0\\ \ecart\dis   C\left(1+\norm{f}_{W^{1,2(1+r)}(\Omega)}\right)&\hbox{if }  -{1\over 2}<r<0.
\ea\right.
\end{equation}
Moreover
\begin{equation}
\label{teo:reg:cond3b}
\hat \sigma\ \hbox{ is parallel to }\nu \hbox{ on }\partial\Om,
\end{equation}
with $\nu$ the unitary outside normal to $\partial\Om.$
\item For $1\leq i,j\leq N$ and $f\in W^{1,1}(\Om)\cap L^2(\Om)$ 
\beq\label{defetaij}\partial_i\hat\theta\hat\sigma_{j}-\partial_j\hat\theta\hat\sigma_i=(1+c\hat\theta)(\partial_j\hat\sigma_i-\partial_i\hat\sigma_j)\Chi_{\{|\hat\sigma|=\hat\mu\}}\in L^2(\Om).\eeq
Moreover, if $\hat\theta$ only takes a finite number of values a.e. in $\Om$, then
\beq\label{detevan}\partial_i\hat\theta\hat\sigma_{j}-\partial_j\hat\theta\hat\sigma_i=0,\quad 1\leq i,j\leq N,\quad {\rm curl}(|\hat\sigma|^{p'-2}\hat\sigma)=0\ \hbox{ in }\Om.\eeq
where, for a distribution from $\Om$ into $\R^N$,  the ${\rm curl}$ operator is defined as ${\rm curl}(\Phi):=\frac{1}{2}\left(\nabla \Phi-\nabla \Phi^{\top} \right)$.
\end{enumerate}
\end{theorem}
\begin{rem} As in \cite{Cas2} we can also obtain some local regularity results for $\hat u$, $\hat\theta$ and $\hat \sigma$ but, for the sake of simplicity, we have preferred to only state and prove the global regularity result.
\end{rem}
\begin{rem} \label{remSoUn} If we assume that $f$ belongs to $W^{1,1}(\Om)\cap L^2(\Om)$, that the unrelaxed problem \eqref{CompMin} has a solution $(\hat u,\hat\theta)$, and that $\Om$ is simply connected, then (\ref{detevan}) proves the existence of $w\in W^{1,p}(\Om)$ such that $\hat\sigma=|\nabla w|^{p-2}\nabla w$ a.e in $\Om$. By \eqref{teo:reg:cond3b}, we must also have $\hat u$ constant in each connected component of $\partial\Om.$ Assuming then that $\partial\Om$ has only a connected component and taking into account that $w$ is defined up to an additive constant, we get
\beq\label{ConsesigUP} \hat\sigma=|\nabla w|^{p-2}\nabla w,\quad w\hbox{ solution of }\ \left\{\ba{l}\dis -{\rm div}\,( |\nabla w|^{p-2}\nabla w)=f\ \hbox{ in }\Om\\ \ecart\dis w=0\ \hbox{ on }\partial\Om.\ea\right.
\eeq
We will show that this implies that the unrelaxed problem has no solution in general.
\end{rem}
 \begin{theorem}
\label{teo:nonexistence}
Let $\Omega\subset\R^N$ be a connected open set of class $C^{1,1}$ with connected boundary and $f=1$. If there exists a solution of \eqref{Problema:Original}, then $\Omega$ is a ball.
\end{theorem}
\begin{rem} In the case $p=2$, Theorem \ref{teo:nonexistence} has been proved in \cite{MuTa} assuming that \eqref{Problema:Original} has a smooth solution and in  \cite{Cas2} in the general case.
\end{rem}
The proof of Theorem \ref{ThprReg} will follow from the following Lemma.
\begin{lemma}\label{lemma:reg} Let $\Omega\subset\R^N$ be a $C^{2}$ bounded open set  and $G:[0,\infty)\to[0,\infty)$ be a $C^1$ function such that there exist $\lambda,\mu>0$ and $p>1$ satisfying
\begin{equation}
\label{lemma:reg:hipd1}
G(s)=s^{p-2},\qquad \forall\,s\geq \mu,
\end{equation}
\begin{equation}
\label{lemma:reg:hipd2}
0\leq G(s)+G'(s)s,\quad G(s)\leq \lambda s^{p-2},\qquad  \forall\,s\geq 0.
\end{equation}
Let  $u\in C^{2}(\overline{\Omega})$ be such that there exists $f\in C^{1,1}(\overline{\Omega}$) satisfying
\begin{equation}
\label{lemma:reg:eq:1}
-{\rm div}\,\Big(G(|\nabla u|)\nabla u\Big)=f \mbox{ in }\Omega,\quad u=0\ \hbox{ on }\partial\Om.
\end{equation}
Then, the following estimates hold:
\begin{enumerate}
\item For every $q\in (p',\infty)$, there exists $C>0$ depending only on $p$, $q$  and $\Om$, such that 
\begin{equation}
\label{lemma:reg:cond1}
\norm{\nabla u}_{L^{q(p-1)}(\Omega)^N}\leq C\Big(\norm{f}_{W^{-1,q}(\Omega)}^{1\over p-1}+\mu\Big).
\end{equation}
\item For every $q>N$ there exists $C>0$ depending only on $p$, $q$  and $\Om$ such that 
\begin{equation}
\norm{\nabla u}_{L^{\infty}(\Omega)^N}\leq C\Big(\norm{f}_{L^q(\Omega)}^{1\over p-1}+\mu\Big).
\label{lemma:reg:cond2}
\end{equation}
 \item For every $\gamma>-1$, there exists $C>0$ depending only on $p,N,\lambda,\gamma$ and $\Omega$ such that 
\begin{equation}
\label{lemma:reg:cond3}
\ba{l}\dis\into |\nabla u|^\gamma\Big({G'\big(|\nabla u|\big)\over |\nabla u|}\big|\nabla^2u\nabla u\big|^2+G\big(|\nabla u|\big)\big|\nabla^2u\big|^2\Big)\,dx\\ \ecart\dis 
\leq C\mu^{p+\gamma}+C\mu^{1+\gamma}\|f\|_{W^{1,1}(\Om)}+C\|f\|_{L^{p+\gamma\over p-1}(\Om)}^{p+\gamma\over p-1},
\ea\qquad \hbox{ if }\gamma\geq p-2,
\end{equation}
\begin{equation}
\label{lemma:reg:cond3b}
\ba{l}\dis\into |\nabla u|^\gamma\Big({G'\big(|\nabla u|\big)\over |\nabla u|}\big|\nabla^2u\nabla u\big|^2+G\big(|\nabla u|\big)\big|\nabla^2u\big|^2\Big)\,dx
\\ \ecart\dis \leq C\mu^{p+\gamma}+C\|f\|_{W^{1,{p+\gamma\over p-1}}(\Om)},\ea\qquad \hbox{ if }-1<\gamma<p-2.
\end{equation}
\end{enumerate}
\end{lemma}
\begin{proof} In order to prove \eqref{lemma:reg:cond1}, we write (\ref{lemma:reg:eq:1}) as
$$-{\rm div} \big(|\nabla u|^{p-2}\nabla u\big)=f-{\rm div} \Big(|\nabla u|^{p-2}\nabla u-G(|\nabla u|)\nabla u\Big)\ \hbox{ in }\Omega,$$
where the last term in the right-hand side is bounded in $W^{-1,\infty}(\Omega)$ by $C\mu^{p-1}$. Then the result follows from Theorem 2.3 in \cite{Mingione}.\par
For the rest of the proof let us differentiate equation \eqref{lemma:reg:eq:1}
with respect to $x_{i}$. This gives 
\begin{equation}
\label{Ecuder}-{\rm div}\Big(L\nabla \partial_{i}u\Big)=\partial_{i} f\ \hbox{ in }\Omega,
\end{equation}
with 
\beq \label{defopL} L=\frac{G'\big(|\nabla u|\big)}{|\nabla u|}\nabla u\otimes\nabla u+G\big(|\nabla u|\big)I.\eeq
Observe that $L$ is non-negative thanks to \eqref{lemma:reg:hipd2}.\par
In order to estimate $\partial_i u$ from (\ref{Ecuder}), we also need to add some boundary conditions. For this purpose, fixed $\bar x\in \partial\Om$, we use that there exist $\delta>0$ and functions $\tau^1,\ldots,\tau^N\in C^1(B(\bar x,\delta))^N$ such that for every $ x \in B(\bar{x},\delta)$
\beq\label{condort} \left\{\ba{l}\dis \big\{\tau^1(x),\ldots,\tau^N(x)\big\}\ \hbox{ is an orthonormal basis of }\R^N,\\ \ecart
\tau^N(x)\hbox{ agrees with the unitary outside normal vector to }\Om\hbox{ on }\partial\Om\cap B(\bar x,\delta).\ea\right.\eeq
Using that
$$\nabla u=\sum_{i=1}^N\big(\nabla u\cdot \tau^i\big)\tau^i\ \hbox{ a.e. in }B(\bar x,\delta),$$
and \eqref{lemma:reg:eq:1}, we get
\beq\label{ecreg1}-\sum_{i=1}^N{\rm div}\big(G(|\nabla u|)\tau^i\big)\nabla u\cdot\tau^i-\sum_{i=1}^N\nabla \big(\nabla u\cdot\tau^i\big)\cdot\tau^iG(|\nabla u|)=f\ \hbox{ in }\Om,\eeq
where  thanks to $u$ vanishing on $\partial\Om$, we have 
$$\nabla u=(\nabla u\cdot\tau^N)\tau^N,\quad \nabla u\cdot\tau^i=0,\quad\nabla(\nabla u\cdot\tau^i)\cdot\tau^i=0\  \hbox{ on }\partial\Om,\ 1\leq i\leq N-1.$$
Thus, developping \eqref{ecreg1}, we get
$$-L\nabla^2 u\tau^N\cdot\tau^N=f+ G(|\nabla u|)\Big({\rm div}\,\tau^NI+\big(\nabla \tau^N\big)^t\Big)\tau^N\cdot\nabla u\ \hbox{ on }\partial\Om\cap B(\bar x,\delta).$$
By the arbitrariness of $\bar x$, we then deduce the existence of a vector function $h\in L^\infty(\partial\Om)^N$, which only depends on $\Om$, such that $\nabla u$ satisfies the boundary conditions
\beq\label{ecreg2} \left\{\ba{l}\dis \nabla u=|\nabla u|s\nu,\quad s\in\{0,1\} \hbox{ a.e. on }\partial\Om,\\ \ecart\dis -L\nabla^2 u\nu\cdot\nu=f+G(|\nabla u|)h\cdot\nabla u\ \hbox{ on }\partial\Om,\ea\right.\eeq
with $\nu$ the unitary outside normal on $\partial\Om.$
 \par
 Let us now prove (\ref{lemma:reg:cond1}). We reason similarly to \cite{DiBe}. For 
 \beq\label{defwthre}w=|\nabla u|^{2},\eeq
 and $k>\mu^p$, we multiply (\ref{Ecuder}) by $\big(w^{\frac{p}{2}}-k\big)^+\partial_iu\in H^{1}(\Omega)$ and integrate by parts. Adding in $i$ and taking into account \eqref{ecreg2}, we get
 $$\ba{l}\dis {p\over 4}\int_{\{w^{p\over 2}\geq k\}} w^{p-2\over 2}L\nabla w\cdot\nabla w\,dx+\sum_{i=1}^N\into \big(w^{\frac{p}{2}}-k\big)^+ L\nabla\partial_iu\cdot\nabla \partial_i u\,dx\\ \ecart\dis =-\int_{\partial\Om} s|\nabla u|\big(f+G(|\nabla u|)h\cdot\nabla u\big)
 \big(w^{\frac{p}{2}}-k\big)^+ds(x)+\into \nabla f\cdot \nabla u\big(w^{\frac{p}{2}}-k\big)^+dx\\ \ecart\dis
 =-\int_{\partial\Om} s|\nabla u|G(|\nabla u|)h\cdot\nabla u
 \big(w^{\frac{p}{2}}-k\big)^+ds(x)-\into f\Delta u\big(w^{\frac{p}{2}}-k\big)^+dx\\ \ecart\dis
 -{p\over 2}\int_{\{w^{p\over 2}\geq k\}} w^{p-2\over 2}f\nabla u\cdot\nabla w\,dx, \ea$$
 which thanks to $k>\mu$, \eqref{lemma:reg:hipd1} and \eqref{defopL} proves
 $$\ba{l}\dis \int_{\{w^{p\over 2}\geq k\}} w^{p-2}|\nabla w|^2dx+\into \big(w^{\frac{p}{2}}-k\big)^+ w^{p-2\over 2}\big|\nabla^2u\big|^2dx\\ \ecart\dis 
 \leq C\int_{\partial\Om} w^{p\over 2}
 \big(w^{\frac{p}{2}}-k\big)^+ds(x)+C\into |f|\big|\nabla^2 u\big|\big(w^{\frac{p}{2}}-k\big)^+dx
 +C\int_{\{w^{p\over 2}\geq k\}} w^{p-1\over 2}|f||\nabla w|dx, \ea$$
 and then, using Young's inequality
 \beq\label{ecreg3}\ba{l}\dis \int_{\{w^{p\over 2}\geq k\}} w^{p-2}|\nabla w|^2dx+\into \big(w^{\frac{p}{2}}-k\big)^+ w^{p-2\over 2}\big|\nabla^2u\big|^2dx\\ \ecart\dis 
 \leq C\int_{\partial\Om} w^{p\over 2}
 \big(w^{\frac{p}{2}}-k\big)^+ds(x)+C\int_{\{w^{p\over 2}\geq k\}} |f|^2w\,dx. \ea\eeq
 In the first term on the right-hand side we use that, thanks to the compact embedding of $W^{1,1}(\Om)$ into $L^1(\partial\Om)$, for every $\ep>0$, there exists $C_\ep>0$ such that
 $$\int_{\partial\Om} |v|ds(x)\leq C_\ep\into |v|dx+\ep\into |\nabla v|dx,\qquad\forall\, v\in W^{1,1}(\Om).$$
 Therefore there exists a constant $C$ depending on $p$ and $\epsilon$ such that
 $$ \int_{\partial\Om} w^{p\over 2}\big(w^{\frac{p}{2}}-k\big)^+ds(x)\leq  C\into w^{p\over 2}\big(w^{\frac{p}{2}}-k\big)^+dx+\ep\int_{\{w^{p\over 2}\geq k\}} w^{p-1}|\nabla w|dx.$$\par

 Replacing this inequality in (\ref{ecreg3}), taking $\ep$ small enough, and using  Young's inequality, we get
$$ \int_{\{w^{p\over 2}\geq k\}} w^{p-2}|\nabla w|^2dx 
 \leq C\int_{\{w^{p\over 2}\geq k\}} w^pdx+C\int_{\{w^{p\over 2}\geq k\}} |f|^2w\,dx,$$
 which by Sobolev's inequality and $f$ in $L^q(\Om)$ provides
  \beq\label{ecreg4}\left(\into\big|\big(w^{p\over 2}-k\big)^+\big|^{2^\ast}dx\right)^{2\over 2^\ast}\leq C\int_{\{w^{p\over 2}\geq k\}} w^pdx+
 C\|f\|_{L^q(\Om)}^2\left(\int_{\{w^{p\over 2}\geq k\}} w^{q\over q-2}dx\right)^{q-2\over q},\eeq
 with
 $$2^\ast={2N\over N-2}\ \hbox{ if }N>2,\quad 2^\ast \in (2,\infty)\ \hbox{ if }N=2.$$
 Now, we use that $q>N$ allows us to take $r>1$ large enough to have
$${2^\ast\over 2}\Big({q-2\over q}-{1\over r}\Big)>1,\qquad {2^\ast\over 2}\Big(1-{p\over r}\Big)>1.$$
For such $r$, we use H\"older's inequality in (\ref{ecreg4}) to get
$$\ba{ll}\dis \left(\into\big|\big(w^{p\over 2}-k\big)^+\big|^{2^\ast}dx\right)^{2\over 2^\ast} &\dis \leq C\left(\into w^rdx\right)^{p\over r}\Big|\{w^{p\over 2}\geq k\}\Big|^{1-{p\over r}}\\ \ecart &\dis+
 C\|f\|_{L^q(\Om)}^2\left(\into w^rdx\right)^{1\over r}\Big|\{w^{p\over 2}\geq k\}\Big|^{{q-2\over q}-{1\over r}}\ea$$
which by (\ref{lemma:reg:cond1}) with $q=2r/(p-1)$ and
$$\|f\|_{W^{-1,{2r\over p-1}}(\Om)}\leq C\|f\|_{L^q(\Om)},$$
implies 
$$ \left(\into\big|\big(w^{p\over 2}-k\big)^+\big|^{2^\ast}dx\right)^{2\over 2^\ast} \leq C\Big(\|f\|_{L^q(\Om)}^{1\over p-1}+\mu\Big)^{2p}
\Big|\{w^{p\over 2}\geq k\}\Big|^{
\min\big(1-{p\over r},{q-2\over q}-{1\over r}\big)
} .$$
 Taking $h>k$ and defining $\varphi$ by 
 $$\varphi(k)=\Big|\{w^{p\over 2}\geq k\}\Big|,$$
 we have then proved
 $$ \varphi(h)^{2\over 2^\ast}\leq {C\big(\|f\|_{L^q(\Om)}^{1\over p-1}+\mu\big)^{2p}\over (h-k)^2}\varphi(k)^{
 \min\big(1-{p\over r},{q-2\over q}-{1\over r}\big)
 },\ \hbox{ for }h>k\geq \mu^p,$$
 where $C$ only depends on $p,N$, and $\Om$. Lemma 4.1 in \cite{Sta} then proves \eqref{lemma:reg:cond2}.
 \par\medskip
 Let us now prove \eqref{lemma:reg:cond3}. Defining $w$ by (\ref{defwthre}), we take $(w+\ep)^{\gamma\over 2}\partial_i u$, with $\ep>0$,  $\gamma>-1$, as test function in  \eqref{lemma:reg:eq:1}. Using \eqref{ecreg2}, we get
\beq\label{Ecrregnu0}\ba{l}\dis {\gamma\over 4}\into (w+\ep)^{\gamma-2\over 2}L\nabla w\cdot\nabla w\,dx+\sum_{i=1}^N\into (w+\ep)^{\gamma\over 2} L\nabla\partial_iu\cdot\nabla \partial_i u\,dx\\ \ecart\dis =-\int_{\partial\Om} s|\nabla u|\big(f+G(|\nabla u|)h\cdot\nabla u\big)
 (w+\ep)^{\gamma\over 2} ds(x)+\into \nabla f\cdot \nabla u (w+\ep)^{\gamma\over 2} dx.
 \ea\eeq
 In this inequality, we observe that the integrand in the left-hand side is nonnegative due to
\beq\label{positiin} \ba{l}\dis 2w\sum_{i=1}^NL\nabla \partial_i u\cdot\nabla\partial_i u-L\nabla w\cdot\nabla w\\ \ecart\dis=
2|\nabla u|^2\sum_{i=1}^NL\nabla \partial_i u\cdot\nabla\partial_i u-2L(\nabla^2u\nabla u)\cdot(\nabla^2 u\nabla u)\geq 0\ \hbox{ a.e. in }\Om,\ea\eeq
 and  $\gamma>-1$. This allows us to use the Fatou Lemma on the left-hand side and the dominated convergence theorem on the right-hand side,
 when $\ep$ tends to zero, to deduce
 \beq\label{ecreg5}\ba{l}\dis {\gamma\over 4}\into w^{\gamma-2\over 2}L\nabla w\cdot\nabla w\,dx+\sum_{i=1}^N\into w^{\gamma \over 2}L\nabla\partial_iu\cdot\nabla \partial_i u\,dx\\ \ecart\dis \leq -\int_{\partial\Om} s|\nabla u|\big(f+G(|\nabla u|)h\cdot\nabla u\big)
 w^{\gamma\over 2} ds(x)+\into \nabla f\cdot \nabla u w^{\gamma\over 2} dx.\ea\eeq\par
 Let us first cosider the case $\gamma\geq p-2$. Defining $T\in W^{1,\infty}(0,\infty)$ by
$$T(s)=\left\{\ba{cl}\dis 0 &\hbox{ if }0\leq s\leq \mu^2\\ \ecart\dis {s\over \mu^2}-1 &\hbox{ if }\mu^2\leq s\leq 2\mu^2 \\ \ecart\dis 1 &\hbox{ if }s\geq 2\mu^2,\ea\right.$$
we decompose the last term in \eqref{ecreg5} as
 $$\into \nabla f\cdot \nabla u \,w^{\gamma\over 2} dx=\into \nabla f\cdot (1-T(w))\nabla u w^{\gamma\over 2} dx+\into \nabla f\cdot T(w)\nabla u w^{\gamma\over 2} dx.$$
Integrating by parts the last term, replacing in \eqref{ecreg5} and using Young's inequality, $h\in L^\infty(\partial\Om)$, and \eqref{lemma:reg:hipd1}, we deduce 
\beq\label{ecreg6}\ba{l}\dis\into w^{\gamma-2\over 2}L\nabla w\cdot\nabla w\,dx+\sum_{i=1}^N\into w^{\gamma \over 2}L\nabla\partial_iu\cdot\nabla \partial_i udx \leq \mu^{1+\gamma}\int_{\partial\Om}|f|\,ds(x)\\ \ecart\dis+C\int_{\partial\Om} w^{p+\gamma\over 2}\,ds(x)+
\mu^{1+\gamma}\into |\nabla f|dx+C\into |f|^2w^{\gamma-p+2\over 2}dx+C\mu^{1+\gamma}\into |f|\,dx.\ea\eeq
For the second term on the right-hand side we use the continuous embedding of $W^{1,1}(\Om)$ into $L^1(\partial\Om)$ and Young's inequality to get
\beq\label{ecreg6b}\ba{l}\dis\int_{\partial\Om} w^{p+\gamma\over 2}\,ds(x)\leq C\mu^{p+\gamma}+\int_{\partial\Om}\big|(w-\mu^2)^+\big|^{p+\gamma\over 2}\,ds(x)\\ \ecart\dis \leq  C\mu^{p+\gamma}+C\into w^{p+\gamma\over 2}dx+C\int_{\{w\geq \mu^2\}} w^{p+\gamma-2\over 2}|\nabla w|\,dx\\ \ecart\dis
\leq  C\mu^{p+\gamma}+C\Big(1+{1\over \delta}\Big)\into w^{p+\gamma\over 2}\,dx+C\delta\int_{\{w\geq \mu^2\}}  w^{p+\gamma-4\over 2}|\nabla w|^2dx,\ea\eeq
with $\delta>0$ arbitrary. Taking $\delta$ small enough, replacing in \eqref{ecreg6} and using H\"older's inequality we have
$$\ba{l}\dis\into w^{\gamma-2\over 2}L\nabla w\cdot\nabla w\,dx+\sum_{i=1}^N\into w^{\gamma \over 2}L\nabla\partial_iu\cdot\nabla \partial_i udx \leq \mu^{1+\gamma}\int_{\partial\Om}|f|\,ds(x)\\ \ecart\dis+ C\mu^{p+\gamma}+C\into w^{p+\gamma\over 2}dx+
\mu^{1+\gamma}\into |\nabla f|dx+C\into |f|^{p+\gamma\over p-1}dx+C\mu^{1+\gamma}\into |f|\,dx.\ea$$
Using \eqref{lemma:reg:cond1} with $q={p+\gamma\over p-1}$ and the continuous imbedding of $L^q(\Om)$ into $W^{-1,q}(\Om),$ combined with
(\ref{positiin}) and
\beq\label{ecreg6c}\sum_{i=1}^N L\nabla\partial_iu\cdot\nabla \partial_i u={G'\big(|\nabla u|\big)\over |\nabla u|}\big|\nabla^2u\nabla u\big|^2+G\big(|\nabla u|\big)\big|D^2u\big|^2,\ \hbox{ a.e. in }\Om,\eeq
we conclude \eqref{lemma:reg:cond3}.
 \par
 We now assume $-1<\gamma< p-2$. In this case we estimate the right-hand side in 
 (\ref{ecreg5}) as follows:\par
 For the first term, using (\ref{ecreg6b}), we have for $\delta<1$
 \beq\label{ecreg7}\ba{l}\dis \left|\int_{\partial\Om} s|\nabla u|\big(f+G(|\nabla u|)h\cdot\nabla u\big)
 w^{\gamma\over 2} ds(x)\right|\leq C \int_{\partial\Om}\big(|f|w^{\gamma+1\over 2}+w^{p+\gamma\over 2}\big)ds(x)\\ \ecart\dis \leq C\int_{\partial\Om}|f|^{p+\gamma\over p-1}ds(x)+C\int_{\partial\Om} w^{p+\gamma\over 2}ds(x)\\ \ecart\dis
 \leq C\int_{\partial\Om}|f|^{p+\gamma\over p-1}ds(x)+ C\mu^{p+\gamma}+{C\over\delta}\into w^{p+\gamma\over 2}dx+C\delta\int_{\{w\geq \mu^2\}}  w^{p+\gamma-4\over 2}|\nabla w|^2dx.\ea\eeq
For the second term on the right-hand side of (\ref{ecreg5}), we just use H\"older's inequality to get
 \beq\label{ecreg7b}\left|\into \nabla f\cdot\nabla u\,w^{\gamma\over 2}dx\right|\leq C\into |\nabla f|^{p+\gamma\over p-1}dx+C\into w^{p+\gamma\over 2}dx.\eeq
 Using \eqref{ecreg7} with $\delta$ small enough, and \eqref{ecreg7b} in \eqref{ecreg5}, and then using \eqref{lemma:reg:cond1} with $q={p+\gamma\over p-1}$, we conclude \eqref{lemma:reg:cond3b}.
\end{proof}
\begin{rem}
Since the constant in the previous theorem only depends on the norm in $L^\infty$ of the first derivative of the functions $\{\tau^{i}\}_{i=1}^{N}$ defined in \eqref{condort}, we can relax the conditions $u\in C^{2}(\bar{\Omega})$ and $\Omega$ of class $C^{2}$ to $u\in C^{1,1}(\bar{\Omega})$ and $\Omega$ of class $C^{1,1}$ by a density argument. 
\end{rem}
\begin{rem} As a simple case, Lemma \ref{lemma:reg} can  be applied to the $p$-Laplacian operator, $G(s)=|s|^{p-2}$. Indeed, since here $\mu=0$ it is simple to check that the proof above does not use the assumption $f\in W^{1,1}(\Om)$ in (\ref{lemma:reg:cond3}). Thus, it shows that for $f\in W^{-1,p'}(\Om)\cap L^{p+\gamma\over p-1}(\Om)$, if $\gamma\geq p-2$ or $f\in W^{-1,p'}(\Om)\cap W^{1,{p+\gamma\over p-1}}(\Om)$ if $-1<\gamma<p-2$, there exists a solution $u$ of \eqref{lemma:reg:eq:1} such that 
$$|\nabla u|^{p+\gamma-2\over 2}|\nabla^2u|\ \hbox{ belongs to }L^2(\Om),$$ 
i.e. $|\nabla u|^{p+\gamma\over 2}$ belongs to $H^1(\Om)$. 
In particular, it proves that $u$ belongs to $H^2(\Om)$ if $p<3$ and $f$ belongs to $W^{1,{2\over p-1}}(\Om)$. This  is a known result which can be found in \cite{CamSciu}. It also proves that for   $f\in  L^{2(1+r)}(\Om)$ if $r\geq 0$, or $f\in W^{1,2(1+r)}(\Om)$ if $-1/2<r<0$
the flux $\sigma=|\nabla u|^{p-2}\nabla u$ satisfies that
$|\sigma|^r D\sigma$ belongs to $L^2(\Om)^{N\times N}$, or equivalently, that $|\sigma|^r\sigma$ belongs to $H^1(\Om)^N$.
The case $r=0$  has been proved in \cite{lou2008singular}. 
\end{rem}
\begin{proof} [Proof of  Theorem \ref{ThprReg}] Let us assume the right-hand side $f$ in \eqref{Relaxed problem:min} smooth enough, which by $\hat u$ solution of (\ref{Problema:solou}) implies that $\hat u\in C^{0,\alpha}(\Om)$ for some $\alpha>0$ (see e.g. \cite{DiBe}) { and satisfies
\beq -{\rm div}\left(\frac{F'(|\nabla \hat{u}|)}{|\nabla\hat{u}|}\nabla \hat{u}\right)=f\ \mbox{ in }\Omega,\quad u\in W_{0}^{1,p}(\Omega).
\label{teo:reg:eqUhat}
\eeq}
\par
For $\ep>0$ small and $F$ defined by \eqref{Definicion:F}, we take $F_{\epsilon}:[0,\infty)\to [0,\infty)$ of class $C^{2}([0,\infty))$ such that for some $k>0$, it satisfies
\beq\label{proFep} \left\{\ba{c}\dis F_\ep(0)=0,\quad F'_\ep(s)\geq {s^{p-1}\over 2(1+c)^{p-1}},\quad \ep\leq F''_\ep(s)\leq  \ep+ks^{p-2},\quad \forall\, s\geq 0,\\ \ecart\dis
F_\ep(s)=F(s),\ \forall\, s\geq (1+c)\hat\mu,\quad \lim_{\ep\to 0}\|F_\ep-F\|_{L^\infty(0,\infty)}=0.\ea\right.\eeq
The existence of this approximation is ensured by Theorem 2.1 and {Remark} 3.1 in \cite{Ghomi}. Then,
we define $u_{\epsilon}$ as the unique solution of
\begin{equation}
\min_{u\in W_{0}^{1,p}(\Omega)\cap L^{2}(\Omega)}\left\{ \int_{\Omega}F_{\varepsilon}(|\nabla u|)dx +{1\over 2}\int_{\Omega}|u-\hat{u}|^2dx  -\int_{\Omega}f\,u\,dx\right\}.
\label{teo:reg:problem:min:epsilon}
\end{equation}
and therefore 
{
\beq
 -{\rm div}\left(\frac{F_{\varepsilon}'(|\nabla u\varepsilon|)}{|\nabla u_{\varepsilon}|}\nabla u_\varepsilon\right)+u_{\varepsilon}-\hat{u}=f\ \hbox{ in }\Om.
\label{teo:reg:problem:min:epsilon:1}
\eeq}
Since 
$$\ba{l}\dis \int_{\Omega}F_{\varepsilon}(|\nabla u_\ep|)dx +{1\over 2}\int_{\Omega}|u_\ep-\hat{u}|^2dx  -\int_{\Omega}fu_\ep\,dx \leq  \int_{\Omega}F_{\varepsilon}(|\nabla \hat u|)dx   -\int_{\Omega}f\hat u\,dx,\ea$$
we have that $u_\ep$ is bounded in $W^{1,p}_0(\Om)\cap L^2(\Om)$ and thus, up to a subsequence, it converges weakly in $W^{1,p}_0(\Om)\cap L^2(\Om)$ to a certain function $u_0$. Taking into account the uniform convergence of $F_\ep$ to $F$, and $F$ convex, we can pass to the limit in the above inequality to deduce
$$\ba{l}\dis \into F(|\nabla u_0|)dx +{1\over 2}\int_{\Omega}|u_0-\hat{u}|^2dx  -\into fu_0\,dx\\ \ecart\dis
\leq\liminf_{\ep\to 0}\left(
\int_{\Omega}F_{\varepsilon}(|\nabla u_\ep|)dx +{1\over 2}\int_{\Omega}|u_\ep-\hat{u}|^2dx  -\int_{\Omega}f\,u_\ep\,dx\right)\\ \ecart\dis\leq  \into F(|\nabla \hat u|)dx  -\into f\hat u\,dx,\ea$$
which combined with $\hat u$ solution of (\ref{Problema:solou}) shows 
$u_0=\hat u$ and
\beq\label{thprreg1}\lim_{\ep\to 0}\into F\big(|\nabla u_\ep|\big)dx=\lim_{\ep\to 0}\into F_\ep\big(|\nabla u_\ep|\big)dx=\into F\big(|\nabla \hat u|\big)dx.\eeq\par
On the other hand, the assumptions of $F_\ep$ imply that
{
$$\sigma_\ep=:{F_{\varepsilon}'\big(|\nabla u_\ep|\big)\over |\nabla u_\ep|}\nabla u_\ep$$}
is bounded in $L^{p'}(\Om)^N$, and then {by \eqref{teo:reg:problem:min:epsilon:1}}, for a subsequence, there exists $\sigma_0\in L^{p'}(\Om)^N$ such that
\beq\label{thprreg2}\sigma_\ep\rightharpoonup \sigma_0\ \hbox{ in }L^{p'}(\Om)^N,\quad {-{\rm div}(\sigma_{0})=f \mbox{ in }\Om}.\eeq
Taking $V\in L^p(\Om)^N$ and using the convexity of $F_\ep$, we have
{
$$\into {F'_\ep(|\nabla u_\ep|)\over|\nabla u_\ep|}\nabla u_\ep\cdot\big(V-\nabla u_\ep\big)dx\leq \into \big(F_\ep(|V|\big)-F_\ep(|\nabla u_\ep|)\big)dx,$$}
which can also be written as
{$$\ba{l}\dis\into \Big({F'_\ep(|\nabla u_\ep|)\over|\nabla u_\ep|}\nabla u_\ep-{F'_\ep(|\nabla \hat u|)\over|\nabla \hat u|}\nabla \hat u\Big)\cdot\nabla (\hat u-u_\ep)\,dx\\ \ecart\dis+
\into {F'_\ep(|\nabla \hat u|)\over|\nabla \hat u|}\nabla \hat u\cdot\nabla (\hat u-u_\ep)\,dx
+\into {F'_\ep(|\nabla u_\ep|)\over|\nabla u_\ep|}\nabla u_\ep\cdot\big(V-\nabla \hat u\big)\,dx
\\ \ecart\dis
\leq \into \big(F_\ep(|V|)-F_\ep(|\nabla u_\ep|)\big)dx.\ea$$}
{From \eqref{teo:reg:eqUhat}}, \eqref{thprreg1} and \eqref{thprreg2} we can pass to the limit in this inequality to deduce 
{$$\into \sigma_0\cdot\big(V-\nabla \hat u\big)\,dx\leq \into \big(F(|V|)-F(|\nabla \hat u|)\big)dx,\quad \forall\, V\in L^p(\Om)^N.$$}
Taking $V=\nabla \hat u+tW$, with $W\in L^p(\Om)^N$, $t>0$, dividing by $t$ and passing to the limit when $t$ tends to zero, we get
{$$\into \sigma_0\cdot W\,dx\leq \into {F'(|\nabla \hat u|)\over|\nabla \hat u|}\nabla \hat u\cdot W\,dx,\quad \forall\, W\in L^p(\Om)^N,$$}
which shows 
$$\sigma_0= {F'(|\nabla \hat u|)\over|\nabla \hat u|}\ \hbox{ a.e. in }\Om.$$
We have thus proved
$$u_\ep\rightharpoonup \hat u\ \hbox{ in }W^{1,p}_0(\Om),\qquad {F'_\ep(|\nabla u_\ep|)\over|\nabla u_\ep|}\nabla u_\ep\rightharpoonup {F'(|\nabla \hat u|)\over|\nabla \hat u|}\nabla \hat u\ \hbox{ in }L^{p'}(\Om)^N.$$
\par
Assuming $\Om\in C^{2,\alpha}$ we can apply for example Theorem 15.12 in \cite{GiTr} to deduce that $u_\ep$ belongs to $C^{2,\alpha}(\overline{\Om})$. On the other hand, we have that $G_\ep\in C^1([0,\infty))$ defined by
$$G_\ep(s)={F'_\ep(s)\over s}\ \hbox{ if }s>0,\quad G_\ep(0)=0,$$
satisfies
$$\ba{l}\dis{G'_\ep \big(|\nabla u_\ep\big)\over |\nabla u_\ep|}\big|\nabla^2u_\ep\nabla u_\ep\big|^2+G_\ep \big(|\nabla u_\ep|\big)\big|\nabla^2u_\ep\big|^2\\ \ecart\dis={F'_\ep (|\nabla u_\ep|)\over |\nabla u_\ep|}\Big(|\nabla u_\ep|^2-{|\nabla^2u_\ep\nabla u_\ep|^2\over|\nabla u_\ep|^2}\Big)+F''_\ep (|\nabla u_\ep|){|\nabla^2u_\ep\nabla u_\ep|^2\over|\nabla u_\ep|^2},\ea$$
while
$$|D\sigma_\ep|^2={F'_\ep (|\nabla u_\ep|)^2\over |\nabla u_\ep|^2}\Big(|\nabla u_\ep|^2-{|\nabla^2u_\ep\nabla u_\ep|^2\over|\nabla u_\ep|^2}\Big)+F''_\ep (|\nabla u_\ep|)^2{|\nabla^2u_\ep\nabla u_\ep|^2\over|\nabla u_\ep|^2},$$
Then, the assumptions of $F_\ep$ imply the existence of a constant $C>0$, which only depends on the constant $k$ in (\ref{proFep}) such that
$$|D\sigma_\ep|^2\leq C\big(\ep+|\nabla u_\ep|^{p-2}\big)\left({G'_\ep \big(|\nabla u_\ep\big)\over |\nabla u_\ep|}\big|\nabla^2u_\ep\nabla u\big|^2+G_\ep \big(|\nabla u_\ep|\big)\big|\nabla^2u_\ep\big|^2\right).$$
Using Lemma \ref{lemma:reg} and 
$$|\nabla u_\ep|\leq 2^{1\over p-1}(1+c)|\sigma_\ep|^{1\over p-1},$$
we conclude \eqref{teo:reg:cond1}, \eqref{teo:reg:cond2} and \eqref{teo:reg:cond3} for $f$ and $\Om$ smooth. The general case follows by an approximation argument.\par
Let us now show \eqref{defetaij}. First, we recall that since we are assuming $f\in W^{1,1}(\Om)\cap L^2(\Om)$, we have $\sigma$ in $H^1(\Om)^N$. Using that \eqref{def:flux} implies
$$\nabla \hat u=(1+c\hat\theta)|\hat\sigma|^{p'-2}\hat\sigma\ \hbox{ a.e. in }\Om,$$
and taking $i,j\in \{1,\ldots,N\}$, and $\Phi\in C^\infty_c(0,\infty)$, such that $\Phi=1$ in a neighborhood of $\hat\mu$, we get in the distributional sense 
\beq\label{expprderte}\ba{l}\dis\partial_j\hat u\partial_i[\Phi(|\hat \sigma|)]-\partial_i\hat u\partial_j[\Phi(|\hat \sigma|)]
=\partial_i \big(\partial_j\hat u\,\Phi(|\hat \sigma|)\big)-\partial_j \big(\partial_i\hat u\,\Phi(|\hat\sigma|)\big)\\ \ecart\dis=
\partial_i \Big((1+c\hat\theta)|\hat\sigma|^{p'-2}\Phi(|\hat \sigma|)\hat\sigma_j\Big)-\partial_j \Big((1+c\hat\theta)|\hat\sigma|^{p'-2}\Phi(|\hat \sigma|)\hat\sigma_i\Big)\\ \ecart\dis
=c\partial_i\hat\theta\,|\hat\sigma|^{p'-2}\Phi(|\hat \sigma|)\hat\sigma_j-c\partial_j\hat\theta\,|\hat\sigma|^{p'-2}\Phi(|\hat \sigma|)\hat\sigma_i\\ \ecart\dis
+(1+c\hat\theta)\Big(\partial_i\big(\Phi(|\hat\sigma|)|\hat\sigma|^{p'-2}\hat\sigma_j\big)-\partial_j\big(\Phi(|\hat\sigma|)|\hat\sigma|^{p'-2}\hat\sigma_i\big)\Big),
\ea\eeq
which using that the support of $\Phi$ is compact and that $\sigma$ belongs to $H^1(\Om)^N$ shows 
\beq\label{regdt1}|\hat\sigma|^{p'-2}\Phi(|\hat \sigma|)\big(\partial_i\hat\theta\,\hat\sigma_j-\partial_j\hat\theta\,\hat\sigma_i\big)\in L^2(\Om).\eeq
Now we recall that
$$\hat\theta=0\ \hbox{ in }\{|\hat\sigma|<\hat\mu\},\quad \hat\theta=1\ \hbox{ in }\{|\hat\sigma|>\hat\mu\}.$$ 
This implies that for every $\Psi\in C^\infty_c((0,\infty)\setminus\{\hat\mu\})$ we have
$$|\hat\sigma|^{p'-2}\Phi(|\hat \sigma|)\big(\partial_i\hat\theta\,\hat\sigma_j-\partial_j\hat\theta\,\hat\sigma_i\big)
=|\hat\sigma|^{p'-2}\Phi(|\hat \sigma|)\big(\partial_i\hat\theta\,\hat\sigma_j-\partial_j\hat\theta\,\hat\sigma_i\big)(1-\Psi(|\hat\sigma|)\big).
$$
By \eqref{regdt1} we can take $\hat\Psi=\hat\Psi_\delta$ with
$$0\leq \hat\Psi_\delta\leq 1,\quad \hat\Psi_\delta(\hat\mu)=0,\quad \hat\Psi_\delta(s)\to 1,\ \forall\, s\not =\hat\mu,$$
to deduce that 
$$|\hat\sigma|^{p'-2}\Phi(|\hat \sigma|)\big(\partial_i\hat\theta\,\hat\sigma_j-\partial_j\hat\theta\,\hat\sigma_i\big)$$
vanishes a.e. in $\{|\hat\sigma|\not=\hat\mu\}$ and then that
$$ |\hat\sigma|^{p'-2}\Phi(|\hat \sigma|)\big(\partial_i\hat\theta\,\hat\sigma_j-\partial_j\hat\theta\,\hat\sigma_i\big)=\hat\mu^{p'-2}\Phi(\hat\mu)\big(\partial_i\hat\theta\,\hat\sigma_j-\partial_j\hat\theta\,\hat\sigma_i\big)\Chi_{\{|\hat\sigma|=\hat\mu\}}.$$
On the other hand, recalling that $\nabla|\hat\sigma|=0$ a.e. in $\{|\hat\sigma|=\hat\mu\}$, we can return to (\ref{expprderte}) to conclude \eqref{defetaij}.\par 
Assertion (\ref{detevan}) now follows from Proposition 2.1 in  \cite{Bern}. which shows that 
$$\partial_i\hat\theta\hat\sigma_{j}-\partial_j\hat\theta\hat\sigma_i\in L^2(\Om),$$
implies
$$\partial_i\hat\theta\hat\sigma_{j}-\partial_j\hat\theta\hat\sigma_i=0\ \hbox{ a.e. in }\{\hat\theta=c\},\quad \forall\, c\in [0,1].$$
\end{proof}\par
\begin{proof} [Proof of Theorem \ref{teo:nonexistence}] 
Let $\hat{\omega}$ a mesurable subset of $\Om$, and $\hat u\in W^{1,p}_0(\Om)$ be such that
$(\chi_{\hat\om},\hat u)$ is a solution of \eqref{Relaxed problem:min} with $\tilde f=f$. 
By Remark \ref{remSoUn}, we have
$$
\left(\alpha\mathcal{X}_{\hat{\omega}}+\beta\mathcal{X}_{\Omega\setminus\hat{\omega}}\right)\nabla \hat{u}= \nabla w,$$
with $w$ the unique solution of
\beq\label{defwce}\left\{\ba{l}\dis 
 -{\rm div}\left(|\nabla w|^{p-2}\nabla w\right)=1 \mbox{ in }\Omega \\ \ecart\dis w\in W^{1,p}_0(\Om).\ea\right.
 \eeq
 Thanks to Theorem 1.1 in \cite{LIEBERMAN19881203} and the fisrt corollary in \cite{DiBe} we know that $w$ is in $C^{1,\beta}(\Omega)$ for some $\beta\in (0,1)$, and  (see \cite{Mor}) that it is analytic in $\{|\nabla w|>0\}$. Using Theorem 1.1 in \cite{lou2008singular} (or Theorem \ref{ThprReg}) we also have that $\hat\sigma=|\nabla w|^{p-2}\nabla w$ is in $H^1(\Om)^N$.  Thus, $-{\rm div}\hat\sigma=0$ a.e. in $\{\hat\sigma=0\}$, which combined with $w$ solution of (\ref{defwce}) implies that $\nabla w\not=0$ a.e. in $\Om$. Analogouly, let us prove that for every $\lambda>0$, the set $\{|\nabla w|=\lambda\}$ has zero measure. For this purpose we observe that  a.e. in $\{|\nabla w|=\lambda\}$, we have
 $$0=\Delta|\nabla w|^p=p\lambda^{p-2}\big(|\nabla^2w|^2+(\Delta \nabla w)\cdot\nabla w\big),$$
 but  a.e. in $\{|\nabla w|=\lambda\}$, we also have
 $$0=\nabla {\rm div}(|\nabla w|^{p-2}\nabla w)=\lambda^{p-2}\nabla \Delta w=\lambda^{p-2}\Delta \nabla w.$$
 Therefore $\nabla^2w=0$ a.e. in $\{|\nabla w|=\lambda\}$, which combined with
 $$-\lambda^{p-2}\Delta w=-{\rm div}(|\nabla w|^{p-2}\nabla w)=1 \ \hbox{ a.e. in }\{|\nabla w|=\lambda\},$$
 implies that the set $\{|\nabla w|=\lambda\}$ has zero measure. Now, we recall that 
  thanks to  \eqref{characterization:theta},  the constant $\hat\mu$ in Theorem \ref{Teo:problema:solou} satisfies
$$\{x\in \Omega:\ |\nabla w| >\hat{\mu}\}\subset\hat \omega\subset \{x\in \Omega:\  |\nabla w| \geq \hat{\mu}\},$$
while Theorem \ref{Teo:problema:solou} implies $ |\hat{\omega}|=\kappa.$
So, using that $ |\{|\nabla w|=\hat\mu\}|=0,$ we get (up to a set of null measure)
\beq\label{caraomCE}\om=\{x\in \Omega:\  |\nabla w| <\hat{\mu}\},\eeq
and $|\hat \om|<|\Om|$.
Then, taking a connected component $O$ of the open set $\{x\in \Omega:\  |\nabla w| >\hat{\mu}\},$ we can repeat the argument in \cite{Cas3} to deduce that $O\Subset \Om$ is  an analytic manifold with connected boundary such that 
\begin{equation}
\left\{\begin{array}{ll}
\displaystyle-{\rm div}\left(|\nabla w|^{p-2}\nabla w\right)=1 \mbox{ in }O  \\
\displaystyle w,\ \frac{\partial w}{\partial\nu} \mbox{ are constant on }\partial O.
\end{array}\right.
\end{equation}
From Serrin's Theorem (\cite{Ser}), this proves that  $O$ is an open ball and that $w$ is a radial function in $O$ with respect to its center. Taking into account the analyticity of  $w$ in $\{|\nabla w|\not =0\}$,  the unique continuation principle shows that $\Omega$ is a ball.
\end{proof}
\section{Conclusion Section}
In the present paper we have studied the optimal design of a two-phase material modeled by the $p$-Laplacian operator posed in a bounded open set  $\Om\subset\R^N$. The goal is  to maximize the potential energy (problem (\ref{Problema:Original})) when we only dispose of a limited amount of the best material. Since the problem has not solution in general, we have obtained a relaxed formulation (problems (\ref{ProRelIn}) and (\ref{ProRelIn2}))  where instead of taking in every point of $\Om$ one of both materials,  we use a microscopic mixture  where the proportion $\theta$ of the best  material  takes values in the whole interval $[0,1]$. This new formulation is obtained using homogenization theory. Reasoning by duality, we have also obtained a new formulation of the minimization problem as a min-max problem (problems (\ref{problem:flux min-max problem}) and (\ref{problem:flux max-min problem})). As a consequence we show that although the relaxed problem has not uniqueness in general, the flux $\hat \sigma$ is unique.\par
The optimal conditions for the relaxed problem show that the state function $\hat u$ is the solution of  a nonlinear Calculus of Variation problem (\ref{Problema:solou}). Since the second derivative of the function $F$ in this problem is not uniformly elliptic, the corresponding Euler-Lagrange equation does not provide in general the existence of second derivatives for $\hat u$. However it allows us to show that if the data es smooth enough then, for every $r>-1/2$, the function $|\hat \sigma|^r\hat\sigma$ is in the Sobolev space $H^1(\Om)^N\cap L^\infty(\Om)^N$. Moreover, the optimal proportion $\hat\theta$ is derivable in the orthogonal directions to $\nabla \hat u$. As an application of these results, we show that the original problem has a solution in a smooth open set $\Om$ with  a connected boundary if and only if $\Om$ is a ball.\par
The results obtained in the present paper extend those obtained by other authors in the case of the Laplacian operator (see e.g. \cite{Cas2},
\cite{ChEv}, \cite{GoKoRe}, \cite{MuTa}).

\newpage
\bibliographystyle{plain}

\end{document}